\documentclass{article}
\usepackage{amsmath,color,amsfonts,amsthm}
\usepackage{amssymb,enumerate,enumitem,verbatim}
\usepackage{hyperref}

\newtheorem{lemma}{Lemma}[section]

\newtheorem{theorem}[lemma]{Theorem}
\newtheorem{prop}[lemma]{Proposition}
\newtheorem{defn}[lemma]{Definition}

\global\long\def\a{\alpha}
\global\long\def\e{\varepsilon}

\global\long\def\N{\mathbb{N}}

\title{Rainbow spanning trees in properly coloured complete graphs}

 \author{
 	J\'ozsef Balogh\thanks{Department of Mathematical Sciences,
 		University of Illinois at Urbana-Champaign, Urbana, Illinois 61801, USA and Trinity College, Cambridge, CB2 1TQ, UK. Email: {\tt jobal@math.uiuc.edu}. Research partially supported by NSF Grant DMS-1500121, Arnold O.~Beckman Research Award (UIUC Campus Research Board 15006) and by the Langan Scholar Fund (UIUC).
 	}
 	, 
 	\quad\quad
 	Hong Liu\thanks{Mathematics Institute and DIMAP, University of Warwick, Coventry, CV4 7AL, UK.  Email: {\tt h.liu.9@warwick.ac.uk}. Research supported by a Leverhulme Trust Early Career Fellowship~ECF-2016-523.}
 	\quad and \quad
 	Richard Montgomery\thanks{Trinity College, Cambridge, CB2 1TQ, UK. Email: {\tt r.h.montgomery@dpmms.cam.ac.uk}.}
 }

\begin{document}

\maketitle
\begin{abstract}
	In this short note, we study pairwise edge-disjoint rainbow spanning trees in properly edge-coloured complete graphs, where a graph is rainbow if its edges have distinct colours. Brualdi and Hollingsworth conjectured that every $K_n$ properly edge-coloured by $n-1$ colours has $n/2$ edge-disjoint rainbow spanning trees. Kaneko, Kano and Suzuki later suggested this should hold for every properly edge-coloured $K_n$. Improving the previous best known bound, we show that every properly edge-coloured $K_n$ contains $\Omega(n)$ pairwise edge-disjoint rainbow spanning trees.
	
	%Independently, and using different methods, Pokrovskiy and Sudakov recently proved the same result.
	Independently, Pokrovskiy and Sudakov recently proved that every properly edge-coloured $K_n$ contains $\Omega(n)$ isomorphic pairwise edge-disjoint rainbow spanning trees.
\end{abstract}

\section{Introduction}\label{intro}
Given a properly edge-coloured $K_n$, a \emph{rainbow spanning tree} is a tree with vertex set $V(K_n)$ whose edges have distinct colours. Each properly edge-coloured $K_n$ clearly contains many rainbow spanning trees -- for example, any $(n-1)$-edge star in $K_n$ is such a tree. How many \emph{edge-disjoint} rainbow spanning trees can we find in any properly edge-coloured $K_n$? Brualdi and Hollingsworth~\cite{BH} conjectured that every $K_n$ properly edge-coloured by $n-1$ colours contains $n/2$ edge-disjoint rainbow spanning trees (see also Constantine~\cite{const}). Note that, in such a colouring, each colour appears on exactly $n/2$ edges to form a monochromatic perfect matching.  Brualdi and Hollingsworth~\cite{BH} showed that there are at least two edge-disjoint rainbow spanning trees in any $K_n$ properly edge-coloured by $n-1$ colours. 

Kaneko, Kano and Suzuki~\cite{KKS} expanded the Brualdi-Hollingsworth conjecture, suggesting that any properly edge-coloured $K_n$ (using any number of colours) should contain $n/2$ edge-disjoint rainbow spanning trees. In such graphs, Kaneko, Kano and Suzuki~\cite{KKS} showed that there are at least three edge-disjoint rainbow spanning trees. Akbari and Alipour~\cite{AA} studied edge-disjoint rainbow spanning trees using weaker conditions still, showing that any edge-coloured $K_n$, where each colour is only constrained to appear at most $n/2$ times, contains at least two edge-disjoint rainbow spanning trees.

Recent progress has seen far more edge-disjoint rainbow spanning trees found in edge-coloured complete graphs. For some small $\e>0$, Horn~\cite{Horn} proved that every properly $(n-1)$-edge-coloured $K_n$ contains at least $\e n$ edge-disjoint rainbow spanning trees. Carraher, Hartke and Horn~\cite{CHH} showed that every edge-coloured $K_n$ where each colour appears at most $n/2$ times contains at least $\lfloor n/(1000\log n)\rfloor$ edge-disjoint rainbow spanning trees. Here, we consider the intermediate of these conditions, and show that any properly edge-coloured $K_n$ contains linearly many edge-disjoint rainbow spanning trees.

\begin{theorem}\label{rainbowmanytrees}
	Every properly edge-coloured $K_n$ contains at least $n/10^{12}$ pairwise edge-disjoint rainbow spanning trees.
\end{theorem}

We note that our methods are much shorter than those previously capable of finding many edge-disjoint rainbow spanning trees. In particular, Section~\ref{sec-welcor} alone shows that any properly $(n-1)$-edge-coloured $K_n$ contains linearly many edge-disjoint rainbow spanning trees.

Essentially, to prove Theorem~\ref{rainbowmanytrees}, we iteratively remove rainbow spanning trees from $K_n$ while ensuring that the remaining minimum degree does not decrease too much and that we do not use up too many colours. We have two cases, depending on the colouring of $K_n$.
We show (in Lemma~\ref{lem-part}) that every properly edge-coloured $K_n$ either contains many colours which appear on linearly many edges or its colours can be grouped into classes which play a similar role. We use different embedding strategies for these cases (in Lemmas~\ref{wellcol} and~\ref{embed-many-orig-rich}).

\medskip

While finishing our work, we discovered that, using different methods, Pokrovskiy and Sudakov~\cite{PS} recently proved that every properly edge-coloured $K_n$ contains at least $n/10^{6}$ edge-disjoint rainbow copies of a certain spanning tree with radius 2. %Note that the methods used are different.

\medskip

\noindent\textbf{Organisation:} The rest of the paper is organised as follows. In Section~\ref{sec-pre}, we present our main lemmas and show that they imply Theorem~\ref{rainbowmanytrees}. These lemmas, Lemmas~\ref{lem-part},~\ref{wellcol} and~\ref{embed-many-orig-rich}, are proved in Sections~\ref{sec-part},~\ref{sec-welcor} and~\ref{sec-almost} respectively.
%%%%%%%%%%%%%%%%%%%%%%%%%%%%%%%%%%%%%%%%%%%%%%%%%%%%%%%%%%%%%%%%%%%%%%%%%%%%%%%%%%%%%%%%%%%%%%%%%%%%%%%
%%%%%%%%%%%%%%%%%%%%%%%%%%%%%%%%%%%%%%%%%%%%%%%%%%%%%%%%%%%%%%%%%%%%%%%%%%%%%%%%%%%%%%%%%%%%%%%%%%%%%%%

\section{Preliminaries}\label{sec-pre}

\noindent\textbf{Notation.} For a graph $G$, denote by $|G|$ the number of vertices in $G$ and by $\nu(G)$ the size of a largest matching in $G$. For each $A\subseteq V(G)$, let $G[A]$ and $G-A$ be the induced subgraph of $G$ on the vertex set $A$ and $V(G)\setminus A$  respectively. When $F$ is a spanning subgraph of $G$, denote by $G\setminus F$ the spanning subgraph of $G$ with edge set $E(G)\setminus E(F)$. For any set of colours $U$ from an edge-coloured graph $G$, let $G_U$ be the maximal subgraph of $G$ whose edges have colour in $U$.
We omit floor and ceiling signs when they are not essential.

\medskip

The main tool we use to find rainbow spanning trees is the following result of Schrijver~\cite{Sch} and Suzuki~\cite{Suzuki}.

\begin{theorem}\label{rain}
Let $G$ be an $n$-vertex edge-coloured graph. If, for every $1\leq s\leq n$ and for every partition $S$ of $V(G)$ into $s$ parts, there are at least $s-1$ edges of different colours between the parts of $S$, then $G$ contains a rainbow spanning tree.\hfill\qed
\end{theorem}

Roughly speaking, we will iteratively find and remove edge-disjoint spanning trees in an edge-coloured $K_n$. By applying Theorem~\ref{rain} to a large subgraph on vertices which have large remaining degree, and using a matching to extend the resulting tree to a spanning tree of $K_n$, we can find spanning rainbow trees with bounded maximum degree in which vertices with small degree in the remaining subgraph of $K_n$ appear as leaves in each new rainbow spanning tree. This allows the iterative removal of spanning trees without decreasing the minimum degree too drastically. This iteration is carried out for the key lemma, Lemma~\ref{wellcoll2}.

The embedding differs depending on the colouring of $K_n$. To describe our cases, we use the following definitions.

\begin{defn}
Let $\alpha>0$. We say a class of colours $U$ is $\alpha$-\emph{rich} in an edge-coloured graph $G$ if $\nu(G_U)\ge \alpha |G|$.
\end{defn}

\begin{defn}
Let $\alpha>0$. We say an edge-coloured graph $G$ is $\alpha$-\emph{well-coloured} if its colours can be partitioned into $|G|-1$ $\alpha$-rich classes.
\end{defn}
\begin{defn}
	Let $\alpha>0$ and $t\in\mathbb{N}$. We say an edge-coloured graph $G$ is $(\alpha,t)$-\emph{robustly-coloured} if, despite the removal of any $t$ rainbow forests, $G$ remains $\a$-well-coloured.
\end{defn}

If a graph does not have many rich colours, then it is robustly coloured, as follows.

\begin{lemma}\label{lem-part} 
	Let $0<\alpha\le 1/8$ and $\beta\leq \a/4$. Then, for any properly edge-coloured $K_n$, either
	%\begin{enumerate} 
	
		(i) there are at least $n-16\beta n$ colours that are $\alpha$-rich, or
		
		(ii) $K_n$ is $\left(\beta,\beta n\right)$-robustly-coloured.
%	\end{enumerate}
\end{lemma}

Let $K_n$ be properly edge-coloured. Letting $\a=1/8$ and $\beta=\a/2400$, apply Lemma~\ref{lem-part}. If \emph{(ii)} in  Lemma~\ref{lem-part} holds for $K_n$, then the following lemma shows that $K_n$ contains at least $n/10^{12}$ edge-disjoint rainbow spanning trees.

\begin{lemma}\label{wellcol} 
	Let $0<\beta\le 1/2$. Then, every $\left(\beta,\beta n\right)$-robustly-coloured $K_n$ contains at least $\beta^2n/2500$ edge-disjoint rainbow spanning trees.
\end{lemma}

If, however, \emph{(i)} in Lemma~\ref{lem-part} holds for $K_n$, then the following lemma shows that $K_n$ contains at least $n/10^{12}$ edge-disjoint rainbow spanning trees.

\begin{lemma}\label{embed-many-orig-rich}
	Let $0<\alpha\le 1/2$. Then, every properly edge-coloured $K_n$ with at least $n-\alpha n/150$ $\alpha$-rich colours contains at least $\alpha^2n/10^6$ edge-disjoint rainbow spanning trees.
\end{lemma}

Thus, in either case, $K_n$ contains at least $n/10^{12}$ edge-disjoint rainbow spanning trees. Therefore, to prove Theorem~\ref{rainbowmanytrees} it is sufficient to prove Lemmas~\ref{lem-part},~\ref{wellcol} and~\ref{embed-many-orig-rich}, which we do in the remaining three sections.

%%%%%%%%%%%%%%%%%%%%%%%%%%%%%%%%%%%%%%%%%%%%%%%%%%%%%%%%%%%%%%%%%%%%%%%%%%%%%%%%%%%%%%%%%%%%%%%%%%%%%%%
%%%%%%%%%%%%%%%%%%%%%%%%%%%%%%%%%%%%%%%%%%%%%%%%%%%%%%%%%%%%%%%%%%%%%%%%%%%%%%%%%%%%%%%%%%%%%%%%%%%%%%%

\section{Partitioning the colours}\label{sec-part}
Before proving Lemma~\ref{lem-part}, we show that any graph without a rich colour, but with many edges, has a rich class of colours, as follows.

\begin{prop}\label{findunited}
Let $\a>0$ and $n\in \N$. Any properly edge-coloured $n$-vertex graph $G$ with at least $4\a n^2$ edges and no $\a$-rich colour has an $\a$-rich class $U$ of colours with $|E(G_U)|\leq 4\a n$.
\end{prop}
\begin{proof}  Let $U$ be a (possibly empty) class of colours which maximises $\nu(G_U)$ subject to $|E(G_U)|\leq 2\nu (G_U)$ and $\nu(G_U)\leq 2\a n$. Assume that $\nu(G_U)<\a n$, for otherwise $U$ satisfies the lemma. Let $M$ be a maximal matching in $G_U$, so that $|M|< \a n$. If, for some $k\geq 1$, a colour $c$ has $k$ edges coloured $c$ in $G-V(M)$ and at most $k$ edges with a vertex in $V(M)$, we could add $c$ to $U$ to contradict the maximality of $U$ (noting that $k\leq \a n$). As there are at most $|V(M)|n$ edges with a vertex in $M$, there are thus at most $2|V(M)|n<4\a n^2$ edges in total, a contradiction.
\end{proof}

\begin{proof}[Proof of Lemma~\ref{lem-part}] Let $C_1,\ldots, C_t$ be the $\a$-rich colours in $K_n$. Let $r=n-1-t$. Suppose $r\geq 16\beta n$, for otherwise we are done. Let $\ell=\beta n$ and let $F\subset K_n$ be the union of any $\ell$ rainbow forests. Let $G$ be the subgraph of $K_n$ formed by removal of the edges of $F$ and of any edges with colour $C_i$ for some $i$. Note that $|E(G)|\geq n(n-1)/2-|E(F)|-t\cdot n/2\geq rn/2-\beta n^2$.

Iteratively, remove as many disjoint $\beta$-rich classes $U_1,\ldots,U_s$ of colours from $G$ as possible, subject to $|E(G_{U_i})|\leq \a n$, and call the resulting subgraph $G'$. Note that, as $G$ has no $\a$-rich colour, $G'$ has no $\beta$-rich colour, $C$ say, for otherwise $\{C\}$ would be $\beta$-rich with $|E(G'_{\{C\}})|\leq \a n$, a contradiction. If $s\geq r$, then, noting that each colour $C_i$ has at least $\a n-\beta n\geq \beta n$ edges outside $F$, these classes along with $C_i$, $1\le i\le t$, demonstrate that $K_n-F$ is $\beta$-well-coloured, and thus, as required, $K_n$ is $(\beta,\ell)$-robustly-coloured. Suppose then, that $s<r$. 

As $G'$ has no $\beta$-rich class of colours with at most $4\beta n\leq \a n$ edges, and no $\beta$-rich colour, by Proposition~\ref{findunited}, $|E(G')|< 4\beta n^2$. On the other hand, as fewer than $r\cdot \a n$ edges were removed by the iterative process, $|E(G')|> r n/2-\beta n^2-\a r n\geq rn/4\geq 4\beta n^2$, a contradiction.
\end{proof}

%%%%%%%%%%%%%%%%%%%%%%%%%%%%%%%%%%%%%%%%%%%%%%%%%%%%%%%%%%%%%%%%%%%%%%%%%%%%%%%%%%%%%%%%%%%%%%%%%%%%%%%
%%%%%%%%%%%%%%%%%%%%%%%%%%%%%%%%%%%%%%%%%%%%%%%%%%%%%%%%%%%%%%%%%%%%%%%%%%%%%%%%%%%%%%%%%%%%%%%%%%%%%%%

\section{Rainbow trees in robustly-coloured graphs}\label{sec-welcor}
We will prove the following more general version of Lemma~\ref{wellcol}, as it will be useful later when proving Lemma~\ref{embed-many-orig-rich}.

\begin{lemma}\label{wellcoll2}
Let $0<\a\le 1/2$, $n\in \N$ and $\ell\leq \alpha^2n/2500$. Let $G$ be an $n$-vertex properly edge-coloured graph. Let  $G_i\subset G$, $1\le i\le \ell$, be $(\alpha,\ell)$-robustly-coloured subgraphs of $G$ satisfying $\delta(G_i)\geq (1-\a/30)n$. Then, $G$ contains edge-disjoint rainbow trees $T_1,\ldots,T_\ell$ so that, for each $1\leq i\leq \ell$, $T_i$ is a spanning tree of $G_i$.
\end{lemma}

Note that Lemma~\ref{wellcol} follows immediately from Lemma~\ref{wellcoll2} by taking $\alpha=\beta$, $\ell=\beta^2 n/2500$ and $G=G_1=\ldots=G_\ell=K_n$. Before proving Lemma~\ref{wellcoll2}, we first show that any well-coloured graph with large minimum degree contains a rainbow spanning tree.

\begin{lemma}\label{embed1} 
	Let $0<\alpha \leq 1/6$, $n\geq 1/2\alpha^2$ and let $G$ be a $2\alpha$-well-coloured $n$-vertex graph, with $\delta(G)\geq (1-\alpha)n$. Then, $G$ contains a rainbow spanning tree.
\end{lemma}
\begin{proof}
	Assume, for contradiction, that $G$ has no rainbow spanning tree. By Theorem~\ref{rain}, for some $k\geq 2$, there is a partition $S=\{S_1,\ldots,S_k\}$ of $V(G)$ with $0<|S_1|\le |S_2|\le\ldots\le |S_k|$ which has at most $k-2$ colours between the sets in $S$. 
	
	Consider a vertex $v\in S_1$. As $|S_1|\le n/k\le n/2$ and $\delta(G)\ge (1-\alpha)n$, there are at least $n/2-\alpha n\geq 2\alpha n$ edges from $v$ to $\cup_{i=2}^k S_i$, all of which have different colours. Thus, $2\alpha n\leq k-2$, so that $k\geq 2\alpha n$. 
In fact, then, $|S_1|\leq n/k\leq 1/2\alpha\leq \alpha n$. Therefore, $v$ has at least $n-2\alpha n$ neighbours in $\cup_{i=2}^k S_i$, and thus $k> (1-2\alpha)n$.
	
	There must then be a set $W$ of more than $(1-4\alpha)n$ singletons in $S$. Any $2\alpha$-rich colour class $U$ has a matching with at least $2\alpha n$ edges, and so at least $4\alpha n$ vertices, one of which must be in $W$. Thus, there is some edge with colour in $U$ across the partition $S$. As $G$ has at least $n-1$ disjoint such $2\alpha$-rich colour classes, there are at least $n-1>k-2$ colours between the sets in $S$, a contradiction.
\end{proof}

We now show that a careful application of Lemma~\ref{embed1} in a well-coloured graph can find a rainbow spanning tree with a maximum degree bound in which a previously chosen small set of vertices has particularly low degree.

\begin{lemma}~\label{treecor} 
	Let $0<\beta\leq \alpha\leq 1/15$ and $n\geq 1/\a^2$.
	Let $G$ be a $6\alpha$-well-coloured $n$-vertex graph with $\delta(G)\geq 3\beta n$. Let $A\subseteq V(G)$ satisfy $|A|\leq \beta n$ and $\delta(G\setminus A)\geq (1-\a)n$. Then, $G$ contains a rainbow spanning tree with maximum degree at most $(1-\a)n$ in which each vertex in $A$ is a leaf.
\end{lemma}
\begin{proof} Choose $B\subset V(G)$ so that $A\subseteq B$ and $|B|= \alpha n$. Greedily, for each $v\in A$, pick a vertex $u_v\in V(G)\setminus A$ and a colour $c_v$ so that $vu_v\in E(G)$ has colour $c_v$ and the vertices and colours chosen are all distinct. This is possible as $G$ is properly coloured with $\delta(G)\geq 3\beta n\ge 3|A|$. Then, for each $v\in B\setminus A$, greedily pick a vertex $u_v\in V(G)\setminus B$ and a colour $c_v$ so that $vu_v\in E(G)$ has colour $c_v$ and the vertices and colours $u_v$, $c_v$, $v\in B$, are all distinct. This is possible as $G$ is properly coloured with $\delta(G\setminus A)\geq (1-\a)n\ge 3|B|$. %Note that for some $v\in B\setminus A$ it is possible that it was chosen as $u_w$ for some $w\in A$.
	
Let $H$ be the graph $G-B$ with any edges of colour $c_v$ removed for each $v\in B$. Any matching has at most $\a n$ edges incident to $B$. Thus, any $6\a$-rich class $U$ of colours in $G$ is a $5\a$-rich class of colours in $G-B$. Furthermore, if such a class $U$ contains no colour $c_v$, $v\in B$, then $U$  is a $5\a$-rich class of colours in $H$. As the colours of $G$ can be partitioned into $n-1$ classes that are each $5\a$-rich, the colours of $H$ have an (inherited) partition into $n-1-|B|=|H|-1$ classes that are $5\a$-rich. Therefore, $H$ is $5\a$-well-coloured. 

Now, as each vertex in $G$ has at most $\a n$ non-neighbours, and we removed edges of $\a n$ colours from $G-B$ to get $H$, $\delta(H)\ge |H|-2\a n\geq (1-5\a/2)|H|$. Hence, by Lemma~\ref{embed1}, $H$ has a rainbow spanning tree, $T$ say.
Adding the edges $u_vv$, $v\in B$, to $T$, we get a rainbow spanning tree of $G$ with maximum degree at most $(|T|-1)+1\leq |H|= (1-\a)n$ in which each vertex in $A$ is a leaf, as desired.
\end{proof}

\begin{proof}[Proof of Lemma~\ref{wellcoll2}] Note the lemma is vacuous unless $\ell\geq 1$, so we may assume $n\geq 2500/\a^2$.
	Greedily, find edge-disjoint rainbow trees $T_1,\ldots, T_{\ell}$ under the rules that, for each $1\leq i\leq \ell$, 
	\begin{itemize}
		\item $\Delta(T_i)\leq (1-\alpha/6)n$,
		\item every vertex in $F_{i}:=\cup_{j<i}T_{j}$ with degree at least $\alpha n/12$ is a leaf in $T_i$, and
                \item $T_i$ is a spanning tree of $G_i$.
	\end{itemize}
	
	This is possible, as follows. Let $1\le i\le \ell$, and suppose we have found $T_1,\ldots,T_{i-1}$. Let $H_i:=G_i-F_i$. From the rules above, each vertex in $F_i$ has degree at most 
$\alpha n/12+(1-\alpha/6)n+\ell \le  (1-\alpha/15)n$. Thus, $\delta(H_i)\geq (1-\a/30)n-(1-\alpha/15)n= \alpha n/30$. Let $\beta=\a/90$, so that $\delta(H_i)\geq 3\beta n$.

As $G_i$ is $(\alpha,\ell)$-robustly-coloured, $H_i$ is $\alpha$-well-coloured. Let $A_i\subseteq V(G_i)$ be the set of vertices in $G_i$ with degree at least $\alpha n/12$ in $F_i$. Then, recalling that $\ell\leq \alpha^2 n/2500$,
	$$|A_i|\leq 2\ell n/(\alpha n/12)\leq \a n/90=\beta n\leq \delta(H_i)/3.$$
Furthermore, $\delta(H_i- A_i)\ge (1-\a/30)n-|A|-\alpha n/12\ge (1-\alpha/6)n$. Applying Lemma~\ref{treecor} with $\alpha/6$, $\beta$, $H_i$ and $A_i$ playing the roles of $\alpha$, $\beta$, $G$ and $A$  gives the required rainbow spanning tree $T_i$ of $H_i$.
\end{proof}

%%%%%%%%%%%%%%%%%%%%%%%%%%%%%%%%%%%%%%%%%%%%%%%%%%%%%%%%%%%%%%%%%%%%%%%%%%%%%%%%%%%%%%%%%%%%%%%%%%%%%%%
%%%%%%%%%%%%%%%%%%%%%%%%%%%%%%%%%%%%%%%%%%%%%%%%%%%%%%%%%%%%%%%%%%%%%%%%%%%%%%%%%%%%%%%%%%%%%%%%%%%%%%%
%%%%%%%%%%%%%%%%%%%%%%%%%%%%%%%%%%%%%%%%%%%%%%%%%%%%%%%%%%%%%%%%%%%%%%%%%%%%%%%%%%%%%%%%%%%%%%%%%%%%%%%
%%%%%%%%%%%%%%%%%%%%%%%%%%%%%%%%%%%%%%%%%%%%%%%%%%%%%%%%%%%%%%%%%%%%%%%%%%%%%%%%%%%%%%%%%%%%%%%%%%%%%%%

\section{Rainbow trees in graphs with many rich colours}\label{sec-almost}

Where $K_n$ has $n-1-r$ many rich colours and $r$ is small, we need to ensure that the iterative removal of rainbow spanning trees does not remove all the non-rich colours. We do this by first finding many edge-disjoint rainbow matchings of $r$ edges using the non-rich colours, reserving one for each rainbow spanning tree we subsequently find. We find the matchings using the following proposition and lemma, before proving Lemma~\ref{embed-many-orig-rich}.

\begin{prop}\label{findmatch}
Let $r,n\in\N$. Any properly edge-coloured $n$-vertex graph $G$, with at least $rn/3$ edges, maximum degree at most $n/10$ and with no $(n/10)$-rich colour, contains a rainbow matching with $r$ edges.
\end{prop}
\begin{proof} Pick a maximal rainbow matching $M$ in $G$. Say $M$ has fewer than $r$ edges, for otherwise we are done. Note that any edge with no vertex in $V(M)$ must have some colour from $M$. Therefore, $G$ has at most $|E(M)|\cdot n/10+|V(M)|\cdot n/10< rn/3$ edges, a contradiction.
\end{proof}

\begin{lemma}\label{findmatch2}
Let $r,n\in\N$ satisfy $r\leq n/100$. Any properly edge-coloured $n$-vertex graph $G$ with minimum degree at least $r$ contains $n/100$ edge-disjoint rainbow matchings with $r$ edges.
\end{lemma}
\begin{proof} Let $\ell=n/100$. Let $C_1,\ldots,C_{s}$ be maximally many $(1/20)$-rich colours in $G$ subject to $s\leq r$, and let $G'$ be the graph $G$ with all edges of these colours removed. Let $v_1,\ldots, v_{s'}$ be maximally many vertices in $G'$ with degree at least $n/20$ subject to $s+s'\leq r$. Let $G''=G'-\{v_1,\ldots,v_s\}$ and $t=r-s-s'\geq 0$. As $\delta(G)\geq r$, $\delta(G'')\geq t$. Greedily, pick edge-disjoint rainbow matchings $M_1,\ldots, M_\ell$ with $t$ edges in $G''$. This is possible, as, for each $0\leq i<\ell$, $G''-M_1-\ldots -M_i$ has at least $t|G''|/2-t\cdot \ell\geq t|G'|/3$ edges, and thus a suitable matching $M_{i+1}$ exists by Proposition~\ref{findmatch}.

Now, greedily, for each $M_i$, add $s'$ edges to $M_i$ by adding, for each $v_j$, an edge $xv_j$ in $E(G')$ with $x\notin V(M_i)$ such that $M_i$ remains rainbow and the matchings $M_1,\ldots, M_\ell$ remain edge-disjoint. For each vertex $v_j$ and matching $M_i$, when we seek to find $x$, at most $\ell$ neighbouring edges of $v_j$ are in the other matchings, at most $2\ell$ neighbours of $v_j$ are in $V(M_i)$, and $M_i$ has at most $\ell$ colours. Thus, there will be at least $n/20-4\ell>0$ choices for $x$. This shows that the greedy process can extend the matchings as described.

Now, greedily, for each $M_i$, add $s$ edges to $M_i$ by adding an edge of each colour $C_j$ such that $M_i$ remains rainbow and the matchings $M_1,\ldots, M_\ell$ remain edge-disjoint. Similarly to the argument above, we will always have at least $n/20-3\ell>0$ choices to add an edge of each colour $C_j$ to each $M_i$. The resulting matchings then satisfy the lemma.
\end{proof}

\begin{proof}[Proof of Lemma~\ref{embed-many-orig-rich}] Let $\ell=\a^2 n/10^6$. Let $C_1,\ldots,C_s$ be the $\alpha$-rich colours of $K_n$, so that $s\geq n-\a n/150$. Let $R\subset K_n$ be the graph with edges without an $\a$-rich colour. Let $r=\max\{n-1-s,0\}\leq \a n/150$. As $\delta(R)\ge r$, by Lemma~\ref{findmatch2}, we can find $\ell$ edge-disjoint rainbow matchings with $r$ edges, $M_1, \ldots,M_\ell$ say. For each $1\le i\le \ell$, let $Z_i\subseteq V(M_i)$ contain an arbitrary vertex from each edge in $M_i$, let $X_i=V(K_m)\setminus Z_i$ and let $U_i$ be the colours of $K_n$ which do not appear in $M_i$.

Let $G=K_n-M_1-\ldots-M_\ell$, so that $\delta(G)\geq n-1-\ell$. For each $1\leq i\leq \ell$, let $G_i=G_{U_i}[X_i]$. Each $\alpha$-rich colour in $K_n$ has all but at most $|Z_i|\leq \a n/150$ of its edges in $K_n-Z_i$, and, given any $\ell$ rainbow forests in $G_i$, at most $\ell\leq\a n/4$ edges of each colour can appear in the forests. Thus, $G_i$ is $(\a n/2,\ell)$-robustly-coloured.

Furthermore, $G_i$ has minimum degree at least $n-1-\ell-2|Z_i|\geq (1-\alpha/60)n$. Therefore, by Lemma~\ref{wellcoll2}, $G$ contains edge-disjoint rainbow trees $T_1$, \ldots, $T_\ell$ so that $T_i$ is a spanning tree of $G_i$. Then, $M_i\cup T_i$, $1\leq i\leq \ell$, is a collection of $\ell$ edge-disjoint rainbow spanning trees in $K_n$.
\end{proof}

%%%%%%%%%%%%%%%%%%%%%%%%%%%%%%%%%%%%%%%%%%%%%%%%%%%%%%%%%%%%%%%%%%%%%%%%%%%%%%%%%%%%%%%%%%%%%%%%%%%%%%%
%%%%%%%%%%%%%%%%%%%%%%%%%%%%%%%%%%%%%%%%%%%%%%%%%%%%%%%%%%%%%%%%%%
%%%%%%%%%%%%%%%%%%%%%%%%%%%%%%%%%%%%%%%%%%%%%%%%%%%%%%%%%%%%%%%%%%%%%%%%

\end{document}